\documentclass[a4paper]{amsart}

\usepackage{amsmath,amssymb,amsthm}
\usepackage{eucal,bm}
\usepackage[colorlinks=true,backref=page]{hyperref}
\usepackage[all]{xy}

\numberwithin{equation}{section}

\SelectTips{eu}{12}

\newtheorem{theorem}{Theorem}[section]
\newtheorem{proposition}[theorem]{Proposition}
\newtheorem{lemma}[theorem]{Lemma}

\theoremstyle{definition}

\theoremstyle{remark}

\numberwithin{equation}{section}

\newcommand{\Z}{\mathbb{Z}}
\newcommand{\Q}{\mathbb{Q}}

\newcommand{\C}{\mathbb{C}}
\newcommand{\Sq}{\operatorname{Sq}}

\SelectTips{cm}{}

\title[Homotopy commutativity in quasitoric manifolds]{Homotopy commutativity in quasitoric manifolds}

\author[S. Hasui]{Sho Hasui}
\address{Department of Mathematics, Osaka Metropolitan University, Sakai, 599- 8531, Japan}
\email{s.hasui@omu.ac.jp}

\author[D. Kishimoto]{Daisuke Kishimoto}
\address{Faculty of Mathematics, Kyushu University, Fukuoka 819-0395, Japan}
\email{kishimoto@math.kyushu-u.ac.jp}

\author[Y. Tong]{Yichen Tong}
\address{Institute of Theoretical Sciences, Westlake Institute of Advanced Study, Westlake University, Zhejiang 310030, China}
\email{tongyichen@westlake.edu.cn}

\author[M. Tsutaya]{Mitsunobu Tsutaya}
\address{Faculty of Mathematics, Kyushu University, Fukuoka 819-0395, Japan}
\email{tsutaya@math.kyushu-u.ac.jp}

\date{\today}



\begin{document}

\maketitle

\begin{abstract}
  We prove that the loop space of a quasitoric manifold is homotopy commutative if and only if the underlying polytope is a product of $3$-simplices $(\Delta^3)^n$ and the characteristic matrix is equivalent to a matrix of certain type. Quasitoric manifolds over $(\Delta^3)^n$ include generalized Bott manifolds, and we also construct an infinite family of homotopy nonequivalent generalized Bott manifolds over $(\Delta^3)^n$, only half of them have homotopy commutative loop spaces. In particular, for each $n\ge 2$, there are infinitely many homotopy types in $6n$-dimensional quasitoric manifolds having homotopy (non)commutative loop spaces.
\end{abstract}

%%%%% Section 1 %%%%%

\section{Introduction}\label{Introduction}

Quasitoric manifolds were introduced by Davis and Januszkiewicz \cite{DJ} as a topological counterpart of smooth projective toric varieties. By definition, a quasitoric manifold is a closed manifold of dimension $2n$ equipped with a locally standard action of $T^n$ such that the orbit space $M/T^n$ is isomorphic to an $n$-dimensional simple polytope as a manifold with corners. Recall that every toric variety is constructed from a fan, a combinatorial object. There is a similar combinatorial construction of quasitoric manifolds, each of which is equivalent (in a precise sense defined in Section \ref{Loop space decomposition}) to that associated to a simple polytope $P$ and a certain characteristic matrix over $P$. Here, we remark that our equivalences of quasitoric manifolds are weaker than those in \cite{DJ} as they respect a fixed isomorphism $M/T^n\cong P$ while ours do not. %Then every quasitoric manifold is constructed from combinatorial data of a simple polytope and a characteristic matrix, up to equivalence.

It is well known that properties of a toric variety are described in terms of the corresponding fan, which exhibits a fascinating connection between algebraic geometry and combinatorics. Then it may be possible to describe topological properties of a quasitoric manifold in terms of the underlying simple polytope and the characteristic matrix, which also exhibits a fascinating connection between topology and combinatorics. There are examples of such descriptions for quasitoric manifolds, cohomology and Chern classes as in \cite{DJ}.

The understanding of a given space goes often through the study of its loop space. A first question is then whether or not it is commutative, up to homotopy. In this paper, we study the homotopy commutativity of the loop space of a quasitoric manifold. See \cite{A,AP,D1,D2,GPTW,IK2,IK3,KMY,Z} for other results on the loop spaces of quasitoric manifolds and related spaces. Complex projective spaces are special quasitoric manifolds, and the homotopy commutativity of their loop spaces were determined by Ganea \cite{G}. The first result completely determines whether or not the loop space of any quasitoric manifold is homotopy commutative in terms of the underlying simple polytope and the characteristic matrix. Let $\Delta^n$ and $E_n$ denote the $n$-simplex and the $n$-dimensional identity matrix.

%\cite{A,AP,D1,D2,GPTW,IK2,IK3,KMY,Z}.

\begin{theorem}
  \label{main}
  The loop space of a quasitoric manifold over a simple polytope $P$ is homotopy commutative if and only if $P=(\Delta^3)^n$ and the characteristic matrix is equivalent to
  \begin{equation}
    \label{main matrix}
    \begin{pmatrix}
      E_3&a_{11}&&a_{12}&&a_{13}&&&a_{1n}\\
      &a_{21}&E_3&a_{22}&&a_{23}&&&a_{2n}\\
      &a_{31}&&a_{32}&E_3&a_{33}&&&a_{3n}\\
      &\vdots&&\vdots&&\vdots&\ddots&&\vdots\\
      &a_{n1}&&a_{n2}&&a_{n3}&&E_3&a_{nn}
    \end{pmatrix}
  \end{equation}
  for $a_{ij}\in\Z^3$ such that
  \begin{equation}
    \label{main condition}
    a_{ii}={}^t(1,1,1)\quad\text{and}\quad(1,1,1)a_{ij}\equiv 0\mod 2\quad(i\ne j),
  \end{equation}
  where the facets of $(\Delta^3)^n$ are ordered as in Section \ref{Quasitoric manifold}.
\end{theorem}

Remarks on Theorem \ref{main} are in order. First, equivalences of characteristic matrices will be defined in Section \ref{Loop space decomposition}. Second, the loop spaces of quasitoric manifolds over a common simple polytope have the same homotopy type. Then Theorem \ref{main} may indicate that there are quasitoric manifolds whose loop spaces are homotopy equivalent but not H-equivalent, which is verified by Theorem \ref{family} below. Third, we can further consider the higher homotopy commutativity of the loop space of a quasitoric manifold if it is homotopy commutative. Actually, by looking at the cohomology of a quasitoric manifold, we can find a nontrivial quadruple higher Whitehead product if its loop space is homotopy commutative. Then if the loop space of a quasitoric manifold is homotopy commutative, it is not a $C_4$-space in the sense of Williams \cite{W}, so it is not very highly homotopy commutative. Fourth, every characteristic matrix over $(\Delta^3)^n$ is equivalent to the matrix \eqref{main matrix} satisfying the first condition of \eqref{main condition} (Proposition \ref{standard form}). Then the second condition of \eqref{main condition} guarantees that the loop space of a quasitoric manifold over $(\Delta^3)^n$ is homotopy commutative. On the other hand, $\C P^n$ is a quasitoric manifold over $\Delta^n$, and in particular, a characteristic matrix of $\C P^3$ is
\begin{equation}
  \label{B}
  B=
  \begin{pmatrix}
    1&0&0&1\\
    0&1&0&1\\
    0&0&1&1
  \end{pmatrix}.
\end{equation}
Then Theorem \ref{main} recovers Ganea's result \cite{G} that the loop space of $\C P^n$ is homotopy commutative if and only if $n=3$, where $\C P^n$ is a quasitoric manifold over $\Delta^n$. Thus Theorem \ref{main} can be thought of as an extension of Ganea's result. See \cite{KMMT,KTT} for other extensions of Ganea's result.

As mentioned above, every characteristic matrix over $\Delta^3$ is equivalent to \eqref{B}, so every quasitoric manifold over $\Delta^3$ is equivalent to $\C P^3$ (cf. \cite[Example 1.18]{DJ}). However, in general, it is quite hard to describe all characteristic matrices over a given simple polytope, and this is the case for $(\Delta^3)^n$ with $n\ge 2$ as in \cite{CMS}. Then one cannot immediately see how many nonequivalent quasitoric manifolds over $(\Delta^3)^n$ for $n\ge 2$ there are, whose loop spaces are (not) homotopy commutative. For each $n\ge 2$, we construct an infinite family of homotopy nonequivalent quasitoric manifolds over $(\Delta^3)^n$, only half of them have homotopy commutative loop spaces. Let $B$ be as in \eqref{B}, and let
\[
  N=\begin{pmatrix}0&0&0&1\\0&0&0&0\\0&0&0&0\end{pmatrix}.
\]
For $k>0$ and $n\ge 2$, we define a $3n\times 4n$ matrix
\[
  B(k,n)=
  \begin{pmatrix}
    B&kN\\
    &B&kN\\
    &&\ddots&\ddots\\
    &&&B&kN\\
    &&&&B
  \end{pmatrix}.
\]
The matrix $B(k,n)$ is a characteristic matrix over $(\Delta^3)^n$, so we get the corresponding quasitoric manifold $M(k,n)$ over $(\Delta^3)^n$. Note that $M(k,n)$ is defined only for $n\ge 2$ and a positive integer $k$. Observe that Theorem \ref{main} shows that the loop space of $M(k,n)$ is homotopy commutative if and only if $k$ is even. The second result concerns the homotopy types in $M(k,n)$, which implies that for each $n\ge 2$, there are infinitely many homotopy types in $6n$-dimensional quasitoric manifolds having homotopy (non)commutative loop spaces.

\begin{theorem}
  \label{family}
  The quasitoric manifolds $M(k,n)$ and $M(l,n)$ are not homotopy equivalent for $k\ne l$.
\end{theorem}

Remarks on Theorem \ref{family} are in order. First, the quasitoric manifold $M(k,n)$ is a so-called generalized Bott manifold, that is, $M(k,n)$ is obtained by iterated ``nice'' $\C P^3$-bundles starting from a point. Second, we will actually prove that $M(k,n)$ satisfies cohomological rigidity; $M(k,n)$ and $M(l,n)$ are equivalent if and only if their integral cohomology are isomorphic. Third, if $M_i$ is a quasitoric manifold over $P_i$ for $i=1,\ldots,n$, then $M_1\times\cdots\times M_n$ is a quasitoric manifold over $P_1\times\cdots\times P_n$. Hence one can construct a quasitoric manifold over $(\Delta^3)^n$ from quasitoric manifolds over $(\Delta^3)^m$ for $m<n$. However, $M(k,n)$ is not equivalent to a product of nontrivial quasitoric manifolds (Theorem \ref{atomicity}), so it is an ``atomic'' quasitoric manifold over $(\Delta^3)^n$.

The paper is organized as follows. In Section \ref{Loop space decomposition}, we recall the basics of quasitoric manifolds, and show a loop space decomposition of quasitoric manifolds. Then by using this decomposition, we deduce that the loop space of a quasitoric manifold is homotopy commutative only if the underlying simple polytope is a product of $\Delta^3$. In Section \ref{Computation of Whitehead products}, we extend the method of Barrat, James and Stein \cite{BJS} for computing Whitehead products. In Section \ref{Quasitoric manifold}, we prove Theorem \ref{main} by applying the results in Sections \ref{Loop space decomposition} and \ref{Computation of Whitehead products}. We also prove Theorem \ref{family} by a direct cohomology computation.

\subsection*{Acknowledgement}

DK and MT were partially supported by JSPS KAKENHI Grant Numbers JP22K03284 and JP22K03317, respectively. YT was partially supported by JST SPRING Grant Number JPMJSP2110. The authors are grateful to J\'{e}r\^{o}me Scherer and anonymous referees for useful comments.

%%%%% Section 2 %%%%%

\section{Loop space decomposition}\label{Loop space decomposition}

In this section, we recall the basic properties of quasitoric manifolds, and show a loop space decomposition of a quasitoric manifold. Then by applying this decomposition, we prove that the loop space of a quasitoric manifold is not homotopy commutative unless its underlying simple polytope is $(\Delta^3)^n$.

First, we define characteristic matrices over a simple polytope and equivalences among them. Let $P$ be an $n$-dimensional convex polytope. A codimension one face of $P$ will be called a facet. We say that $P$ is simple if exactly $n$ facets of $P$ meet at each vertex. For example, simplices are simple polytopes, and a product of simple polytopes is a simple polytope. Suppose that $P$ is simple and has $m$ facets $F_1,\ldots,F_m$. A characteristic matrix over $P$ is an integer matrix $(a_1\;\ldots\;a_m)$ for $a_1,\ldots,a_m\in\Z^n$ such that $\det(a_{i_1}\;\ldots\;a_{i_n})=\pm 1$ whenever $F_{i_1}\cap\cdots\cap F_{i_n}\ne\emptyset$ for $i_1<\cdots<i_n$. Since an automorphism of $P$ as a combinatorial polytope permutes facets, it acts on characteristic matrices over $P$ by column permutation. We define that characteristic matrices $A$ and $B$ over $P$ are equivalent if
\begin{equation}
  \label{equivalence matrix}
  A=\alpha\cdot(QBD)
\end{equation}
for $Q\in\mathrm{GL}_n(\Z)$, a diagonal matrix $D$ with diagonal entries $\pm 1$ and an automorphism $\alpha$ of $P$.

Next, we recall the construction of a quasitoric manifold using a moment-angle complex. Let $K$ be a simplicial complex with vertex set $[m]=\{1,2,\ldots,m\}$, where an ordering of vertices is given. The moment-angle complex for $K$ is defined by
\[
  Z_K=\bigcup_{\sigma\in K}Z(\sigma),
\]
where $Z(\sigma)=X_1\times\cdots\times X_m$ such that $X_i=D^2$ for $i\in\sigma$ and $X_i=S^1$ for $i\not\in\sigma$. Note that the $m$-dimensional torus $T^m$ acts naturally on $Z_K$. We will use the following obvious property of a moment-angle complex. For $\emptyset\ne I\subset[m]$, let
\[
  K_I=\{\sigma\in K\mid\sigma\subset I\}.
\]

\begin{lemma}
  \label{retract}
  For $\emptyset\ne I\subset[m]$, $Z_{K_I}$ is a retract of $Z_K$.
\end{lemma}

\begin{proof}
  We can identify $Z_{K_I}$ with the subspace
  \[
    \{(x_1,\ldots,x_m)\in Z_K\mid x_i\text{ is the basepoint for }i\in I\}
  \]
  of $Z_K$. Then the statement follows.
\end{proof}

Let $P$ be an $n$-dimensional simple polytope with $m$ facets. Let $K(P)$ denote the boundary of the dual simplicial polytope of $P$. Then $K(P)$ is an $(n-1)$-dimensional simplicial sphere with $m$ vertices. Let $A$ be a characteristic matrix over $P$. Then the kernel of the linear map $A\colon\Z^m\to\Z^n$ defines a split subtorus $T(A)$ of dimension $m-n$ which acts freely on $Z_{K(P)}$. Let
\[
  M(A)=Z_{K(P)}/T(A).
\]
By \cite{DJ}, we have:

\begin{proposition}
  \label{M(A)}
  The orbit space $M(A)$ is a quasitoric manifold over $P$ such that every quasitoric manifold over $P$ is equivalent to $M(A)$ for some characteristic matrix $A$ over $P$.
\end{proposition}

Let $M$ and $N$ be quasitoric manifolds of dimension $2n$. A map $f\colon M\to N$ is weakly equivariant if there is an automorphism $\theta\colon T^n\to T^n$ such that
\[
  f(tx)=\theta(t)f(x)
\]
for $t\in T^n$ and $x\in M$. We say that $M$ and $N$ are equivalent if there is a weakly equivariant homeomorphism between them. Note that if $M$ and $N$ are equivalent, their underlying simple polytopes are isomorphic. Then equivalent quasitoric manifolds are essentially the same. As remarked in Section \ref{Introduction}, our equivalences of quasitoric manifolds are weaker than those in \cite{DJ} as Davis and Januszkiewicz demand equivalences to preserve an extra structure, a fixed isomorphism between $M/T^n$ and a simple polytope. By \cite{DJ}, we also have:

\begin{proposition}
  \label{equivalence}
  The quasitoric manifolds $M(A_1)$ and $M(A_2)$ over $P$ are equivalent if and only if the characteristic matrices $A_1$ and $A_2$ are equivalent as in \eqref{equivalence matrix}.
\end{proposition}

Now we prove a loop decomposition of a quasitoric manifold.

\begin{proposition}
  \label{decomposition}
  Let $M$ be a quasitoric manifold over an $n$-dimensional simple polytope $P$ with $m$ facets. Then there is a homotopy equivalence
  \[
    \Omega M\simeq T^{m-n}\times\Omega Z_{K(P)}.
  \]
\end{proposition}

\begin{proof}
  By Proposition \ref{M(A)}, there is a homotopy fibration $Z_{K(P)}\to M\to BT^{m-n}$, so we get an H-fibration
  \[
    \Omega Z_{K(P)}\to\Omega M\to T^{m-n}.
  \]
  By \cite[Theorem 3.4.7]{BP}, $Z_{K(P)}$ is $2$-connected, so $\Omega Z_{K(P)}$ is simply-connected. Then the map $\Omega M\to T^{m-n}$ has a section, implying the above H-fibration splits. Thus the statement is proved.
\end{proof}

We record an obvious fact about homotopy commutativity.

\begin{lemma}
  \label{retract commutative}
  Let $X,Y$ be H-groups, and let $f\colon X\to Y$ be an H-map. If $X$ is not homotopy commutative and $f$ has a left homotopy inverse, then $Y$ is not homotopy commutative.
\end{lemma}

\begin{proof}
  Let $g\colon Y\to X$ be a left homotopy inverse of $f$. If $Y$ is homotopy commutative, then by definition, the Samelson product $\langle f,f\rangle$ is trivial, implying a contradiction
  \[
    0\ne\langle 1_X,1_X\rangle=g\circ f\circ\langle 1_X,1_X\rangle=g\circ\langle f,f\rangle=0.
  \]
  Thus the statement is proved.
\end{proof}

Now we consider conditions on the underlying polytope of a quasitoric manifold $M$ that guarantee $\Omega M$ is not homotopy commutative. Let $K$ be a simplicial complex. We say that a nonempty subset $I$ of the vertex set of $K$ is a minimal nonface of $K$ if $I$ is not a simplex of $K$ and all proper subsets of $I$ are simplices of $K$. Equivalently, $K_I=\partial\Delta^{|I|-1}$.

\begin{lemma}
  \label{minimal nonface}
  Let $P$ be a simple polytope. If $K(P)$ has a minimal nonface of cardinality $2,3$ or $\ge 5$, then the loop space of a quasitoric manifold over $P$ is not homotopy commutative.
\end{lemma}

\begin{proof}
  By Proposition \ref{decomposition} and Lemma \ref{retract commutative}, it suffices to show $\Omega Z_{K(P)}$ is not homotopy commutative. Let $I$ be a minimal nonface of $K(P)$ of cardinality $k$. Then $Z_{K(P)_I}=Z_{\partial\Delta^{k-1}}=S^{2k-1}$, so by Lemma \ref{retract}, $S^{2k-1}$ is a retract of $Z_{K(P)}$. By \cite{Ad}, the Whitehead product $[1_{S^{2k-1}},1_{S^{2k-1}}]$ is nontrivial for $k=3$ and $k\ge 5$, so by the adjointness of Whitehead products and Samelson products \cite{S}, $\Omega S^{2k-1}$ is not homotopy commutative for $k=3$ and $k\ge 5$. Thus by Proposition \ref{decomposition} and Lemma \ref{retract commutative}, $\Omega Z_{K(P)}$ is not homotopy commutative either.

  Now we suppose $k=2$. Let $M$ be a quasitoric manifold over $P$. Since $S^3$ is a retract of $Z_{K(P)}$, $H^3(Z_{K(P)};\Q)$ has a basis $\{ u_1,\ldots,u_l\}$ for some $l\ge 1$. By Proposition \ref{M(A)}, there is a homotopy fibration
  \[
    Z_{K(P)}\to M\to BT^{m-n},
  \]
  where $m$ is the number of facets of $P$ and $n=\dim P$. By \cite[Theorem 3.4.7]{BP}, $Z_{K(P)}$ is 2-connected, so in the Serre spectral sequence of the above homotopy fibration, each $u_i$ is transgressive. Moreover, by \cite[Proposition 3.10]{DJ}, $H^*(M;\Q)$ is generated by elements of degree two, so the transgression images of $u_1,\ldots,u_l$ are linearly independent. Then we get
  \[
    H^*(M;\Q)=\Q[t_1,\ldots,t_{m-n}]/(q_1,\ldots,q_l),\quad|t_i|=2
  \]
  for $*\le 5$, where $q_i$ is the transgression image of $u_i$. This readily implies that the minimal Sullivan model for $M$ is given by
  \[
    (\Q[t_1,\ldots,t_{m-n}]\otimes\Lambda(x_1,\ldots,x_l),d),\quad dt_i=0,\quad dx_i=q_i
  \]
  in dimension $\le 4$. Since $|q_i|=4$, $q_i$ is a quadratic polynomial in $t_1,\ldots,t_{m-n}$. Thus by \cite[Proposition 13.16]{FHT}, $M$ has nontrivial Whitehead product, implying $\Omega M$ is not homotopy commutative. Therefore the proof is finished.
\end{proof}

\begin{lemma}
  \label{nondisjoint}
  Let $P$ be a simple polytope. If $K(P)$ has intersecting distinct minimal nonfaces, then the loop space of a quasitoric manifold over $P$ is not homotopy commutative.
\end{lemma}

\begin{proof}
  Let $I_1,I_2$ be minimal nonfaces of $K(P)$ with $I_1\ne I_2$ and $I_1\cap I_2\ne\emptyset$. Let $|I_1\cap I_2|=j>0$ and $|I_k|=i_k+j$ for $k=1,2$. Then for $k=1,2$,
  \[
    Z_{K(P)_{I_k}}=S^{2(i_k+j)-1}
  \]
  Let $\iota_k\colon Z_{K(P)_{I_k}}\to Z_{K(P)_{I_1\cup I_2}}$ denote the inclusion, and let $v_k$ be a generator of $H^{2(i_k+j)-1}(Z_{K(P)_{I_k}})\cong\Z$. Then by Lemma \ref{retract}, there is $u_k\in H^{2(i_k+j)-1}(Z_{K(P)_{I_1\cup I_2}})$ satisfying $\iota_k^*(u_k)=v_k$ for $k=1,2$. Now we assume that the Whitehead product $[\iota_1,\iota_2]$ is trivial. Then there is a homotopy commutative diagram
  \[
    \xymatrix{
      Z_{K(P)_{I_1}}\vee Z_{K(P)_{I_2}}\ar[r]^(.57){\iota_1+\iota_2}\ar[d]&Z_{K(P)_{I_1\cup I_2}}\ar@{=}[d]\\
      Z_{K(P)_{I_1}}\times Z_{K(P)_{I_2}}\ar[r]^(.57)\mu&Z_{K(P)_{I_1\cup I_2}}.
    }
  \]
  Hence $\mu^*(u_1)=v_1\times 1$ and $\mu^*(u_2)=1\times v_2$, so
  \[
    \mu^*(u_1u_2)=\mu^*(u_1)\mu^*(u_2)=v_1\times v_2\ne 0.
  \]
  Thus we get $u_1u_2\ne 0$. On the other hand, since $K(P)_{I_1\cup I_2}$ has at least two minimal nonfaces, it is not a full simplex, implying $\dim Z_{K(P)_{I_1\cup I_2}}\le 2(i_1+i_2+j)-1$. Then $|u_1u_2|=2(i_1+i_2+j)-2+2j>2(i_1+i_2+j)-1\ge \dim Z_{K(P)_{I_1\cup I_2}}$ as $j>0$, so we get $u_1u_2=0$, a contradiction. Thus the Whitehead product $[\iota_1,\iota_2]$ is nontrivial, so $\Omega Z_{K(P)_{I_1\cup I_2}}$ is not homotopy commutative. Therefore by Lemma \ref{retract commutative}, $\Omega Z_{K(P)}$ is not homotopy commutative too, completing the proof.
\end{proof}

Now we are ready to prove:

\begin{proposition}
  \label{polytope}
  Let $M$ be a quasitoric manifold over a simple polytope $P$. If the loop space of $M$ is homotopy commutative, then $P=(\Delta^3)^n$.
\end{proposition}

\begin{proof}
  Suppose $\Omega M$ is homotopy commutative. Then by Lemmas \ref{minimal nonface} and \ref{nondisjoint}, minimal nonfaces of $K(P)$ are of cardinality $4$ and pairwise disjoint. Then
  \[
    K(P)=\underbrace{\partial\Delta^3\star\cdots\star\partial\Delta^3}_n\star\Delta^l
  \]
  for some $l\ge -1$, where $\Delta^{-1}=\{\emptyset\}$. Since $K(P)$ is a simplicial sphere, we have $l=-1$, so
  \[
    K(P)=\underbrace{\partial\Delta^3\star\cdots\star\partial\Delta^3}_n.
  \]
  Thus we obtain $P=(\Delta^3)^n$, as stated.
\end{proof}

We further consider a condition equivalent to the loop space of a quasitoric manifold over $(\Delta^3)^n$ being homotopy commutative. As in the proof of Proposition \ref{polytope}, if $P=(\Delta^3)^n$, then $K(P)$ is the join of $n$ copies of $\partial\Delta^3$, implying
\[
  Z_{K(P)}=(S^7)^n.
\]
Let $M$ be a quasitoric manifold over $(\Delta^3)^n$. Then by Proposition \ref{decomposition}, there is a homotopy equivalence
\[
  \Omega M\simeq(S^1)^n\times(\Omega S^7)^n
\]
which is not necessarily an H-equivalence. For $i=1,\ldots,n$, let $a_i\colon S^1\to\Omega M$ and $b_i\colon S^6\to\Omega M$ be the composite maps
\[
  S^1\xrightarrow{g_i}(S^1)^n\to\Omega M\quad\text{and}\quad S^6\xrightarrow{E}\Omega S^7\xrightarrow{g_i}(\Omega S^7)^n\to\Omega M,
\]
where $g_i$ and $E$ denote the $i$-th inclusion and the suspension map, respectively.

\begin{lemma}
  \label{equivalent condition}
  Let $M$ be a quasitoric manifold over $(\Delta^3)^n$. The loop space of $M$ is homotopy commutative if and only if the Samelson products $\langle a_i,b_j\rangle$ for $i,j=1,\ldots,n$ are trivial.
\end{lemma}

\begin{proof}
  For $i=1,\ldots,n$, let $\bar{b}_i\colon\Omega S^7\to\Omega M$ denote the composite of the $i$-th inclusion $\Omega S^7\to(\Omega S^7)^n$ and the natural map $(\Omega S^7)^n\to\Omega M$. By \cite[Proposition 1]{KK}, $\Omega M$ is homotopy commutative if and only if the Samelson products $\langle a_i,a_j\rangle,\,\langle a_i,\bar{b}_j\rangle,\,\langle\bar{b}_i,\bar{b}_j\rangle$ for $i,j=1,\ldots,n$ are trivial. Clearly, $\langle a_i,a_j\rangle$ are trivial. By \cite[Lemma 2.1]{G}, $\langle a_i,\bar{b}_j\rangle=0$ if and only if $\langle a_i,b_j\rangle=0$, and $\langle\bar{b}_i,\bar{b}_j\rangle=0$ if and only if $\langle b_i,b_j\rangle=0$. Note that each $b_i\colon S^6\to\Omega M$ lifts to a map $\tilde{b}_i\colon S^6\to(\Omega S^7)^n$. Then the Samelson products $\langle b_i,b_j\rangle$ in $\Omega M$ lift to the Samelson prodsucts $\langle\tilde{b}_i,\tilde{b}_j\rangle$ in $(\Omega S^7)^n$. Hence since $(\Omega S^7)^n$ is homotopy commutative, $\langle\tilde{b}_i,\tilde{b}_j\rangle$ are trivial, implying so are $\langle b_i,b_j\rangle$. Thus the proof is finished.
\end{proof}

%%%%% Section 3 %%%%%

\section{Computation of Whitehead products}\label{Computation of Whitehead products}

In this section, we extend the method of Barrat, James, and Stein \cite{BJS} computing Whitehead products. The coefficients of cohomology will be the integers $\Z$.

Let $X$ be a simply-connected finite complex satisfying a homotopy fibration
\begin{equation}
  \label{fibration}
  (S^{2d-1})^n\xrightarrow{\phi}X\xrightarrow{\pi}(\C P^\infty)^n
\end{equation}
for $d\ge 3$ such that
\[
  H^*(X)=\Z[t_1,\ldots,t_n]/(q_1,\ldots,q_n),\quad|t_i|=2,\,|q_i|=2d,
\]
where for $i=1,\ldots,n$, $t_i$ corresponds to the fundamental class of the $i$-th $\C P^\infty$ in $(\C P^\infty)^n$ and $q_i\in\Z[t_1,\ldots,t_n]$ is the transgression image of a generator of $H^{2d-1}(S^{2d-1})$ for the $i$-th $S^{2d-1}$ in $(S^{2d-1})^n$.

\begin{lemma}
  \label{regular sequence}
  The sequence $q_1,\ldots,q_n$ in $\Z[t_1,\ldots,t_n]$ is regular.
\end{lemma}

\begin{proof}
  For any field $\mathbb{F}$, $\mathbb{F}[t_1,\ldots,t_n]$ is Cohen-Macaulay, and the Krull dimension of $H^*(X)\otimes\mathbb{F}$ is zero as $X$ is a finite complex. Then the sequence $q_1,\ldots,q_n$ is regular in $\mathbb{F}[t_1,\ldots,t_n]$ for any field $\mathbb{F}$, so the sequence $q_1,\ldots,q_n$ is regular in $\Z[t_1,\ldots,t_n]$ too, as stated.
\end{proof}

We consider the cofiber $Y$ of the map $\phi\colon(S^{2d-1})^n\to X$. For $i=1,\ldots,n$, let $\beta_i\colon S^{2d-1}\to X$ be the composite
\[
  S^{2d-1}\xrightarrow{i\text{-th incl}}(S^{2d-1})^n\xrightarrow{\phi}X.
\]
Then by degree reasons,
\begin{equation}
  \label{cell Y}
  Y_{4d-2}=X_{4d-2}\cup_{\beta_1}e^{2d}\cup_{\beta_2}\cdots\cup_{\beta_n}e^{2d},
\end{equation}
where $Y_k$ denotes the $k$-skeleton of $Y$.

\begin{lemma}
  \label{Y}
  For $*\le 2d+3$,
  \[
    H^*(Y)=\Z[t_1,\ldots,t_n]/(t_iq_j\mid i,j=1,\ldots,n).
  \]
\end{lemma}

\begin{proof}
  Let $u_i\in H^{2d-1}((S^{2d-1})^n)$ denote the generator corresponding to the $i$-th $S^{2d-1}$ in $(S^{2d-1})^n$. By Lemma \ref{regular sequence}, the elements $q_1,\ldots,q_n$ of $\Z[t_1,\ldots,t_n]$ are linearly independent, so we may assume
  \[
    \tau(u_i)=q_i
  \]
  for $i=1,\ldots,n$, where $\tau$ denotes the transgression in the Serre spectral sequence for the homotopy fibration \eqref{fibration}, implying
  \[
    \delta(u_i)=\pi^*(q_i)
  \]
  for the connecting map $\delta\colon H^{*-1}((S^{2d-1})^n)\to H^*(X,(S^{2d-1})^n)$ of the long exact sequence for the pair $(X,(S^{2d-1})^n)$ and the map
  \begin{equation}
    \label{pi}
    \pi^*\colon H^*((\C P^\infty)^n)\to H^*(X,(S^{2d-1})^n).
  \end{equation}
  By degree reasons, the kernel of the composite
  \[
    H^{2d}((\C P^\infty)^n)\xrightarrow{\pi^*}H^{2d}(X,(S^{2d-1})^n)\to H^{2d}(X)
  \]
  is generated by $q_1,\ldots,q_n$. Then the map \eqref{pi} for $*=2d$ is an isomorphism. Thus since $\widetilde{H}^*(Y)\cong H^*(X,(S^{2d-1})^n)$, it follows from \eqref{cell Y} that the map \eqref{pi} is an isomorphism for $1\le *\le 2d$. On the other hand, the $(2d+2)$-dimensional part of the ideal $(q_1,\ldots,q_n)$ in $\Z[t_1,\ldots,t_n]$ is generated by $t_iq_j$ for $i,j=1,\ldots,n$. Then by \eqref{cell Y}, the proof is finished.
\end{proof}

Let $\bar{\pi}\colon Y\to(\C P^\infty)^n$ denote an extension of the map $\pi\colon X\to(\C P^\infty)^n$. Then by \cite[Theorem 1.1]{G0}, the homotopy fiber of $\bar{\pi}\colon Y\to(\C P^\infty)^n$ has the homotopy type of the join $(S^1)^n\star(S^{2d-1})^n$. We consider the cofiber $\overline{Y}$ of the fiber inclusion of $\bar{\pi}$. Let $g_i\colon A\to A^n$ denote the $i$-th inclusion for $i=1,\ldots,n$, and let $\gamma_{ij}$ denote the composite
\[
  S^{2d+1}=S^1\star S^{2d-1}\xrightarrow{g_i\star g_j}(S^1)^n\star(S^{2d-1})^n\xrightarrow{\text{incl}}Y.
\]
Then the map
\[
  \bigvee_{i,j=1}^ng_i\star g_j\colon\bigvee_{i,j=1}^nS^{2d+1}\to (S^1)^n\star(S^{2d-1})^n
\]
is an inclusion of the $(2d+1)$-skeleton and has a left homotopy inverse, implying
\begin{equation}
  \label{cell barY}
  \overline{Y}_{2d+2}=Y_{2d+2}\cup_{\gamma_{11}}e^{2d+2}\cup_{\gamma_{12}}\cdots\cup_{\gamma_{nn}}e^{2d+2}.
\end{equation}

\begin{lemma}
  \label{barY}
  For $*\le 2d+2$,
  \[
    H^*(\overline{Y})=\Z[t_1,\ldots,t_n].
  \]
\end{lemma}

\begin{proof}
  Since $(S^1)^n\star(S^{2d-1})^n$ is homotopy equivalent to a wedge of spheres, all $(2d+3)$-cells of $\overline{Y}$ are attached to $Y_{2d+2}$. Then by Lemma \ref{Y} and \eqref{cell barY}, the $(2d+3)$-cells of $\overline{Y}$ do not kill any cohomology class of $\overline{Y}_{2d+2}$, implying $H^*(\overline{Y})=H^*(\overline{Y}_{2d+2})$ for $*\le 2d+2$. Now by Lemma \ref{regular sequence}, $t_iq_j$ for $i,j=1,\ldots,n$ are linearly independent in $\Z[t_1,\ldots,t_n]$. Then by arguing as in the proof of Lemma \ref{Y}, the statement is proved.
\end{proof}

\begin{lemma}
  \label{homotopy Y}
  The homotopy group $\pi_{2d+1}(Y)$ is a free abelian group generated by $\gamma_{ij}$ for $i,j=1,\ldots,n$.
\end{lemma}

\begin{proof}
  The statement follows from the homotopy exact sequence of the homotopy fibration $(S^1)^n\star(S^{2d-1})^n\to Y\to(\C P^\infty)^n$, where the $(2d+1)$-skeleton of $(S^1)^n\star(S^{2d-1})^n$ is described as above.
\end{proof}

For $i=1,\ldots,n$, let $\alpha_i\colon S^2\to X$ be a map whose Hurewicz image is the dual of $t_i$, and let $\bar{\beta}_i\colon(D^{2d},S^{2d-1})\to(Y,X)$ denote the obvious extension of $\beta_i\colon S^{2d-1}\to X$. Then
\[
  \delta(\bar{\beta}_i)=\beta_i
\]
for the connecting homomorphism $\delta\colon\pi_*(Y,X)\to\pi_{*-1}(X)$. We consider the relative Whitehead product $[\alpha_i,\bar{\beta}_j]\in\pi_{2d+1}(Y,X)$. See \cite{BM} for the definition. By \cite[(3.5)]{BM},
\[
  \delta([\alpha_i,\bar{\beta}_j])=-[\alpha_i,\beta_j].
\]
Let $\bar{\eta}\colon(D^{2d+1},S^{2d})\to(D^{2d},S^{2d-1})$ be the obvious extension of the Hopf map $\eta\colon S^{2d}\to S^{2d-1}$. By \cite[Theorem (1.4)]{J} (cf. \cite[(5.8)]{T}), we can compute $\pi_{2d+1}(Y,X)$ as follows.

For a commutative ring $R$, let $R\{a_1,\ldots,a_k\}$ denote the free $R$-module with a basis $\{a_1,\ldots,a_k\}$.

\begin{lemma}
  \label{homotopy (Y,X)}
  $\pi_{2d+1}(Y,X)=\Z\{[\alpha_i,\bar{\beta}_j]\mid i,j=1,\ldots,n\}\oplus\Z_2\{\bar{\beta}_i\circ\bar{\eta}\mid i=1,\ldots,n\}$.
\end{lemma}

Let $\beta=\beta_1\vee\cdots\vee\beta_n\colon(S^{2d-1})^{\vee n}\to X$ and $\bar{\beta}=\bar{\beta}_1\vee\cdots\vee\bar{\beta}_n\colon((D^{2d})^{\vee n},(S^{2d-1})^{\vee n})\to(Y,X)$. Let $\iota\colon(A,*)\to(A,B)$ and $\rho\colon A\to A/B$ denote the inclusion and the pinch map, respectively. There is a commutative diagram
\begin{equation}
  \label{pinch}
  \xymatrix{
  &\pi_{2d+1}((D^{2d})^{\vee n},(S^{2d-1})^{\vee n})\ar[r]^(.6)\delta_(.6)\cong\ar[d]^{\bar{\beta}_*}&\pi_{2d}((S^{2d-1})^{\vee n})\ar[d]^{\beta_*}\\
  \pi_{2d+1}(Y)\ar[r]^{\iota_*}\ar[d]^{\rho_*}&\pi_{2d+1}(Y,X)\ar[r]^(.55)\delta\ar[d]^{\rho_*}&\pi_{2d}(X)\\
  \pi_{2d+1}(Y/X)\ar@{=}[r]&\pi_{2d+1}(Y/X)
  }
\end{equation}
in which the middle row is exact. By the homotopy exact sequence for the homotopy fibration \eqref{fibration}, we can see the map $\beta_*$ is an isomorphism. Then there is $\epsilon_{ij}\in\pi_{2d+1}((D^{2d})^{\vee n},(S^{2d-1})^{\vee n})$ such that
\begin{equation}
  \label{epsilon def}
  \beta_*\circ\delta(\epsilon_{ij})=-[\alpha_i,\beta_j],
\end{equation}
implying
\[
  \delta([\alpha_i,\bar{\beta}_j]-\bar{\beta}_*(\epsilon_{ij}))=-[\alpha_i,\beta_j]-\beta_*\circ\delta(\epsilon_{ij})=0.
\]
Hence by Lemma \ref{homotopy Y},
\begin{equation}
  \label{epsilon,zeta}
  [\alpha_i,\bar{\beta}_j]-\bar{\beta}_*(\epsilon_{ij})=\iota_*(\zeta_{ij})
\end{equation}
such that $\zeta_{ij}$ is a linear combination of $\gamma_{kl}\in\pi_{2d+1}(Y)$ for $k,l=1,\ldots,n$.

\begin{lemma}
  \label{epsilon}
  The Whitehead product $[\alpha_i,\beta_j]$ vanishes if and only if $\rho_*(\zeta_{ij})=0$, where $\rho\colon Y\to Y/X$ denotes the pinch map.
\end{lemma}

\begin{proof}
  Since $\beta_*$ in \eqref{pinch} is an isomorphism, it follows from \eqref{epsilon def} that $[\alpha_i,\beta_j]=0$ if and only if $\epsilon_{ij}=0$. By \eqref{cell Y},
  \[
    (Y/X)_{4d-2}=\underbrace{S^{2d}\vee\cdots\vee S^{2d}}_n,
  \]
  so $\rho_*\circ\bar{\beta}_*$ in \eqref{pinch} is an isomorphism. Then $\epsilon_{ij}=0$ if and only if $\rho_*\circ\bar{\beta}_*(\epsilon_{ij})=0$. On the other hand,  $\rho_*([\alpha_i,\bar{\beta}_j])=0$ as $\rho_*(\alpha_i)=0$, so by \eqref{epsilon,zeta}, we get
  \[
    \rho_*\circ\bar{\beta}_*(\epsilon_{ij})+\rho_*(\zeta_{ij})=\rho_*([\alpha_i,\bar{\beta}_j])=0.
  \]
  Thus $\rho_*\circ\bar{\beta}_*(\epsilon_{ij})=0$ if and only if $\rho_*(\zeta_{ij})=0$, completing the proof.
\end{proof}

Now we are ready to prove:

\begin{proposition}
  \label{criterion}
  The Whitehead products $[\alpha_i,\beta_j]$ are trivial for $i,j=1,\ldots,n$ if and only if $\Sq^2q_k=0$ in $\Z_2[t_1,\ldots,t_n]$ for all $k$.
\end{proposition}

\begin{proof}
  By the homotopy fibration \eqref{fibration}, we can see that $\pi_{2d+1}(X)$ is a finite group, so the map $\iota_*$ in \eqref{pinch} is injective by Lemma \ref{homotopy Y}. In particular, $\mathrm{Im}\,\iota_*$ is a free abelian group. On the other hand, by Lemma \ref{homotopy (Y,X)} and \eqref{epsilon,zeta}, the subgroup $A$ of $\pi_{2d+1}(Y,X)$ generated by $\iota_*(\zeta_{ij})$ for $i,j=1,\ldots,n$ is a maximal free abelian subgroup of $\pi_{2d+1}(Y,X)$. Then since $A\subset\mathrm{Im}\,\iota_*$, we obtain $A=\mathrm{Im}\,\iota_*$, implying that $\rho_*(\zeta_{ij})=0$ for $i,j=1,\ldots,n$ if and only if $\rho_*(\gamma_{ij})=0$ for $i,j=1,\ldots,n$.

  By Lemma \ref{Y}, the $(2d)$-cells in \eqref{cell Y} correspond to $q_1,\ldots,q_n$, and by \eqref{cell Y} and \eqref{cell barY},
  \[
    (\overline{Y}/X)_{2d+2}=(\underbrace{S^{2d}\vee\cdots\vee S^{2d}}_n)\cup_{\rho_*(\gamma_{11})}e^{2d+2}\cup_{\rho_*(\gamma_{12})}\cdots\cup_{\rho_*(\gamma_{nn})}e^{2d+2}
  \]
  such that the $(2d+2)$-cells may be considered to be corresponding to $t_iq_j$ for $i,j=1,\ldots,n$. Then as the generator of $\pi_{2d+1}(S^{2d})\cong\Z_2$ is detected by $\Sq^2$, we get that $\rho_*(\gamma_{ij})=0$ if and only if $\Sq^2q_k$ does not include the terms $t_iq_j$ in $H^*(\overline{Y}/X;\Z_2)$ for $k=1,\ldots,n$. Note that in $H^*(\overline{Y};\Z_2)$, $\Sq^2q_k$ must belong to the ideal $(q_1,\ldots,q_n)$ and every degree $2d+2$ element of $(q_1,\ldots,q_n)$ is a linear combination of $t_iq_j$ for $i,j=1,\ldots,n$. Then since the natural map $H^*(\overline{Y}/X;\Z_2)\to H^*(\overline{Y};\Z_2)$ is injective for $2d\le *\le 2d+2$, the above condition on $\Sq^2q_k$ in $H^*(\overline{Y}/X;\Z_2)$ is equivalent to that $\Sq^2q_k=0$ in $H^*(\overline{Y};\Z_2)=\Z_2[t_1,\ldots,t_n]$ ($*\le 2d+2$) for $k=1,\ldots,n$, completing the proof.
\end{proof}

%%%%% Section 4 %%%%%

\section{Quasitoric manifolds over $(\Delta^3)^n$}\label{Quasitoric manifold}

In this section, we prove Theorems \ref{main} and \ref{family}. We fix an ordering of the facets of $(\Delta^3)^n$ as in \cite{CMS} to consider characteristic matrices over $(\Delta^3)^n$. Let $F_1,F_2,F_3,F_4$ be the facets of $\Delta^3$, where any choice of ordering will do by symmetry. Then facets of $(\Delta^3)^n$ are
\[
  F_{ij}=(\Delta^3)^{i-1}\times F_j\times(\Delta^3)^{n-i}
\]
for $i=1,\ldots,n$ and $j=1,2,3,4$. We fix an ordering of facets as
\[
  F_{11},\,F_{12},\,F_{13},\,F_{14},\,F_{21},\,F_{22},\,F_{23},\,F_{24},\,\ldots,\,F_{n1},\,F_{n2},\,F_{n3},\,F_{n4},
\]
where this ordering is used in Theorems \ref{main} and \ref{family}.

\begin{lemma}
  \label{standard form}
  Every characteristic matrix over $(\Delta^3)^n$ is equivalent to a matrix
  \[
    \begin{pmatrix}
      E_3&a_{11}&&a_{12}&&a_{13}&&&a_{1n}\\
      &a_{21}&E_3&a_{22}&&a_{23}&&&a_{2n}\\
      &a_{31}&&a_{32}&E_3&a_{33}&&&a_{3n}\\
      &\vdots&&\vdots&&\vdots&\ddots&&\vdots\\
      &a_{n1}&&a_{n2}&&a_{n3}&&E_3&a_{nn}
    \end{pmatrix}
  \]
  for $a_{ij}\in\Z^3$ such that $a_{ii}={}^t(1,1,1)$ for $i=1,\ldots,n$.
\end{lemma}

\begin{proof}
  Let $B=(b_1^1\;b_2^1\;b_3^1\;b_4^1\;b_1^2\;b_2^2\;b_3^2\;b_4^2\;\cdots\;b_1^n\;b_2^n\;b_3^n\;b_4^n)$ be a characteristic matrix over $(\Delta^3)^n$, where $b_j^i\in\Z^{3n}$. Since the facets of $(\Delta^3)^n$ except for $F_{14},F_{24},\ldots,F_{n4}$ meet at a vertex, the matrix $Q=(b_1^1\;b_2^1\;b_3^1\;b_1^2\;b_2^2\;b_3^2\;\cdots\;b_1^n\;b_2^n\;b_3^n)$ is invertible, so $B$ is equivalent to
  \[
    Q^{-1}B=
    \begin{pmatrix}
      E_3&c_{11}&&c_{12}&&c_{13}&&&c_{1n}\\
      &c_{21}&E_3&c_{22}&&c_{23}&&&c_{2n}\\
      &c_{31}&&c_{32}&E_3&c_{33}&&&c_{3n}\\
      &\vdots&&\vdots&&\vdots&\ddots&&\vdots\\
      &c_{n1}&&a_{n2}&&c_{n3}&&E_3&c_{nn}
    \end{pmatrix}
  \]
  for $c_{ij}\in\Z^3$. Since the facets of $(\Delta^3)^n$ except for $F_{14},\ldots,F_{i-1,4},F_{ij},F_{i+1,4},\ldots,F_{n4}$ meet at a vertex for $i=1,\ldots,n$ and $j=1,2,3$,
  \[
    \det
    \begin{pmatrix}
      E_{3(i-1)}\\
      &C_{ij}\\
      &&E_{3(n-i)}
    \end{pmatrix}
    =\pm 1
  \]
  for $i=1,\ldots,n$ and $j=1,2,3$, where $C_{ij}$ is the $3\times 4$ matrix $(E_3\;c_{ii})$ with $j$-th column removed. Then we get $c_{ii}={}^t(\pm 1,\pm 1,\pm 1)$. Since multiplying columns and rows of a characteristic matrix by $-1$ yields an equivalent characteristic matrix, we obtain that $Q^{-1}B$ is equivalent to the matrix in the statement, completing the proof.
\end{proof}

Now we are ready to prove Theorem \ref{main}.

\begin{proof}
  [Proof of Theorem \ref{main}]
  By Propositions \ref{M(A)} and \ref{polytope}, we only need to consider a quasitoric manifold $M(A)$ over $(\Delta^3)^n$ such that $A$ is a characteristic matrix over $(\Delta^3)^n$ in Lemma \ref{standard form}. By \cite[Theorem 4.14]{DJ},
  \[
    H^*(M(A))=\Z[t_{ij}\mid i=1,\ldots,n,\,j=1,2,3,4]/I+J,\quad|t_{ij}|=2,
  \]
  where $I=(t_{i1}t_{i2}t_{i3}t_{i4}\mid i=1,\ldots,n)$ and
  \[
    J=\left(t_{ij}+\sum_{k=1}^na_{ik}^jt_{k4}\;\middle|\;i=1,\ldots,n,\,j=1,2,3\right),
  \]
  where $a_{ik}={}^t(a_{ik}^1,a_{ik}^2,a_{ik}^3)$. So we get
  \begin{equation}
    \label{H(M(A))}
    H^*(M(A))=\Z[t_1,\ldots,t_n]/(q_1,\ldots,q_n),\quad|t_i|=2
  \end{equation}
  such that
  \[
    q_i=t_i\prod_{j=1}^3\left(\sum_{k=1}^na_{ik}^jt_k\right),
  \]
  where we put $t_i=t_{i4}$. Now
  \begin{align*}
    \Sq^2q_i&=\left(t_i+\sum_{j=1}^3\sum_{k=1}^na_{ik}^jt_k\right)q_i\\
    &=((1+a_{ii}^1+a_{ii}^2+a_{ii}^3)t_i+\sum_{k\ne i}(a_{ik}^1+a_{ik}^2+a_{ik}^3)t_k)q_i\\
    &=\sum_{k\ne i}(a_{ik}^1+a_{ik}^2+a_{ik}^3)t_kq_i
  \end{align*}
  because $a_{ii}={}^t(1,1,1)$. Thus by Lemma \ref{regular sequence} and Proposition \ref{criterion}, the Whitehead products $[\alpha_i,\beta_j]$ are trivial for $i,j=1,\ldots,n$ if and only if $(1,1,1)a_{ij}=a_{ij}^1+a_{ij}^2+a_{ij}^3\equiv 0\mod 2$ for all $i\ne j$. On the other hand, by the adjointness of Whitehead products and Samelson products \cite{S}, the Whitehead product $[\alpha_i,\beta_j]$ is trivial if and only if the Samelson product $\langle a_i,b_j\rangle$ is trivial, where $a_i$ and $b_j$ are as in Section \ref{Computation of Whitehead products}. Therefore by Lemma \ref{equivalent condition}, the proof is finished.
\end{proof}

Hereafter, let $k$ be a positive integer. For $n\ge 1$, we define a graded algebra
\[
  H(k,n)=\Z[t_1,\ldots,t_n]/(t_1^4+kt_1^3t_2,\ldots,t_{n-1}^4+kt_{n-1}^3t_n,t_n^4),\quad|t_i|=2.
\]
We need the following properties of $H(k,n)$.

\begin{lemma}
  \label{x^4}
  If $x\in H(k,n)$ satisfies $|x|=2$ and $x^4=0$, then $x=at_n$ for some $a\in\Z$.
\end{lemma}

\begin{proof}
  Since $|x|=2$, we may put $x=a_1t_1+\cdots+a_nt_n$ for $a_1,\ldots,a_n\in\Z$. Note that $\{t_{i_1}t_{i_2}t_{i_3}t_{i_4}\mid 1\le i_1\le i_2\le i_3\le i_4\le n,\,i_1<i_4\}$ is a basis of the degree $8$ part of $H(k,n)$. We express $x^4$ as a linear combination of this basis. Then $x^4$ includes the term $6a_i^2a_j^2t_i^2t_j^2$ for $i\ne j$, implying $a_ia_j=0$ for $i\ne j$. This readily implies $x=a_it_i$ for some $1\le i\le n$. On the other hand, $t_i^4=0$ in $H(k,n)$ if and only if $i=n$. Thus the statement is proved.
\end{proof}

\begin{lemma}
  \label{atomic H(k,n)}
  If connected graded rings $A,B$ have nontrivial elements of degree two, then $H(k,n)$ is not isomorphic to $A\otimes B$.
\end{lemma}

\begin{proof}
  We prove the statement by induction on $n$. For $n=1$, the statement holds because the degree two part of $H(k,1)$ is isomorphic to $\Z$. We assume the $n=m$ case, and prove the $n=m+1$ case. Suppose that there is an isomorphism $f\colon H(k,m+1)\to A\otimes B$. We may put $f(t_{m+1})=a+b$ for $a\in A$ and $b\in B$, where $|a|=|b|=2$. Then
  \[
    0=f(t_{m+1}^4)=a^4+4a^3b+6a^2b^2+4ab^3+b^4.
  \]
  Since $H(k,m+1)\cong A\otimes B$, $A$ and $B$ are isomorphic to polynomial rings in degrees $<8$, implying $a=0$ or $b=0$. We may assume $b=0$. Then $f$ induces an isomorphism
  \[
    \bar{f}\colon H(k,m+1)/(t_{m+1})\to(A/(a))\otimes B.
  \]
  Since $H(k,m+1)/(t_{m+1})\cong H(k,m)$, it follows from the assumption that $A/(a)=\Z$. So $t=f^{-1}(\bar{f}(t_m))$ and $t_{m+1}$ are linearly independent in $H(k,m+1)$, and $t^4=t_{m+1}^4=0$. This is a contradiction by Lemma \ref{x^4}, and therefore $H(k,m+1)$ is not isomorphic to $A\otimes B$, completing the proof.
\end{proof}

For the rest of the paper, we set $n\ge 2$. Let $M(k,n)$ denote the quasitoric manifold in Theorem \ref{family}. Then by \eqref{H(M(A))},
\begin{equation}
  \label{M(A) H(k,n)}
  H^*(M(k,n))=H(k,n).
\end{equation}
We remark that $M(k,n)$ is a generalized Bott manifold such that $M(k,n+1)$ is the projectivization of a complex vector bundle $E\oplus\underline{\C}^3\to M(k,n)$, where $E$ is the complex line bundle with total Chern class $c(E)=1+kt_1$ and $\underline{\C}$ denotes the trivial bundle. We show atomicity of $M(k,n)$ with respect to products of quasitoric manifolds.

\begin{proposition}
  \label{atomicity}
  The quasitoric manifold $M(k,n)$ is not homotopy equivalent to a product of two nontrivial quasitoric manifolds.
\end{proposition}

\begin{proof}
  Suppose $M(k,n)\simeq M\times N$ for nontrivial quasitoric manifolds $M,N$. Then by the K\"{u}nneth formula,
  \[
    H(k,n)\cong H^*(M)\otimes H^*(N).
  \]
  Thus the statement follows from Lemma \ref{atomic H(k,n)}.
\end{proof}

Now we start to prove Theorem \ref{family}. The following lemma is immediate from Lemma \ref{x^4}.

\begin{lemma}
  \label{f(t_n)}
  Every graded algebra isomorphism $f\colon H(k,n)\xrightarrow{\cong}H(l,n)$ satisfies
  \[
    f(t_n)=\pm t_n.
  \]
\end{lemma}

%\begin{proof}
%  Let $f(t_n)=a_1t_1+\cdots+a_nt_n$ for $a_i\in\Z$. Then we have
%  \begin{align*}
%    0=f(t_n^4)=f(t_n)^4&=\sum_{i=1}^{n-1}(4a_i^3a_{i+1}-ka_i^4)t_i^3t_{i+1}+4\sum_{i\ne j,j-1}a_i^3a_jt_i^3t_j+6\sum_{i<j}a_i^2a_j^2t_i^2t_j^2\\
%    &\quad+12\sum_{\substack{i\ne j,k\\j<k}}a_i^2a_ja_kt_i^2t_jt_k+24\sum_{i<j<k<l}a_ia_ja_ka_lt_it_jt_kt_l.
%  \end{align*}
%  Clearly, $t_i^3t_j$ ($i\ne j$), $t_i^2t_j^2$ ($i<j$), $t_i^2t_jt_k$ ($i\ne j,\,j<k$) and $t_it_jt_kt_l$ ($i<j<k<l$) are linearly independent in $H(l,n)$. Then, in particular, we get $a_ia_j=0$ for $i\ne j$, implying $f(t_n)=a_it_i$ for some $i=1,\ldots,n$. Since $t_i^4=0$ in $H(l,n)$ if and only if $i=n$, we obtain $f(t_n)=a_nt_n$, and since the map $f$ is an isomorphism, $a_n=\pm 1$, completing the proof.
%\end{proof}

For $j=1,\ldots,n$, we define an ideal of $H(k,n)$ by
\[
  I_j(k,n)=(t_{n-j+1},t_{n-j+2},\ldots,t_n).
\]
Then we get a sequence
\[
  I_1(k,n)\subset I_2(k,n)\subset\cdots\subset I_n(k,n).
\]

\begin{lemma}
  \label{I(k,n)}
  Every graded algebra isomorphism $f\colon H(k,n)\xrightarrow{\cong}H(l,n)$ satisfies
  \[
    f(I_j(k,n))=I_j(l,n)
  \]
  for $j=1,\ldots,n$.
\end{lemma}

\begin{proof}
  We show $f(I_j(k,n))=I_j(l,n)$ by induction on $j$. For $j=1$, $f(I_1(k,n))=I_1(l,n)$ by Lemma \ref{f(t_n)}. Assume that the statement holds for $j=1,\ldots,p$. Then the map $f$ induces an isomorphism
  \[
    \bar{f}\colon H(k,n)/I_p(k,n)\xrightarrow{\cong}H(l,n)/I_p(l,n).
  \]
  On the other hand, there is a natural isomorphism
  \[
    H(m,n)/I_p(m,n)\cong H(m,n-p)
  \]
  for any positive integer $m$ such that $I_1(m,n-p)$ in $H(m,n-p)$ lifts to $I_{p+1}(m,n)$ in $H(m,n)$. By the induction hypothesis, $\bar{f}(I_1(k,n-p))\subset I_1(k,n-p)$ through the above natural isomorphism. Thus we get $f(I_{p+1}(k,n))=I_{p+1}(l,n)$, completing the proof.
\end{proof}

\begin{proposition}
  \label{H(k,n)}
  $H(k,n)\cong H(l,n)$ if and only if $k=l$.
\end{proposition}

\begin{proof}
  The if part is trivial, and we consider the only if part. First, we consider the $n=2$ case. Suppose there is an isomorphism $f\colon H(k,2)\xrightarrow{\cong}H(l,2)$. By Lemma \ref{I(k,n)},
  \[
    f(t_1)=\epsilon_1(t_1+ct_2)\quad\text{and}\quad f(t_2)=\epsilon_2t_2
  \]
  for $\epsilon_1,\epsilon_2=\pm 1$ and an integer $c$, so we get
  \begin{align*}
    0&=f(t_1^4+kt_1^3t_2)\\
    &=(4c-l+k\epsilon_1\epsilon_2)t_1^3t_2+(6c^2+3k\epsilon_1\epsilon_2c)t_1^2t_2^2+(4c^3+3k\epsilon_1\epsilon_2c^2)t_1t_2^3.
  \end{align*}
  Then since $t_1^3t_2,\,t_1^2t_2^2,\,t_1t_2^3$ are linearly independent in $H(l,n)$, we obtain
  \[
    6c^2+3k\epsilon_1\epsilon_2c=0,\quad 4c^3+3k\epsilon_1\epsilon_2c^2=0,\quad 4c-l+k\epsilon_1\epsilon_2=0.
  \]
  By the first two equations, we get $c=0$, so by the third equation, we obtain $k=l$, as desired.

  Next, we consider the $n>2$ case. Let $H(m)$ denote the subalgebra of $H(m,n)$ generated by $t_{n-1}$ and $t_n$. Then there is a canonical isomorphism
  \[
    H(m)\cong H(m,2).
  \]
  Suppose there is an isomorphism $f\colon H(k,n)\xrightarrow{\cong}H(l,n)$. Then by Lemma \ref{I(k,n)}, the map $f$ restricts to an isomorphism
  \[
    H(k)\xrightarrow{\cong}H(l).
  \]
  Thus by the $n=2$ case, we get $k=l$, completing the proof.
\end{proof}

We are ready to prove Theorem \ref{family}.

\begin{proof}
  [Proof of Theorem \ref{family}]
  The first statement follows from Theorem \ref{main}, and the second statement follows from \eqref{M(A) H(k,n)} and Proposition \ref{H(k,n)}.
\end{proof}

\end{document}